\documentclass[a4paper,USenglish]{lipics}
\usepackage{amsmath}
\usepackage{amsthm}
\usepackage{amssymb}
\usepackage{bussproofs}

\newcommand{\bn}{\mathbb{N}}
\newcommand{\bool}{\mathbb{B}}
\newcommand{\option}[1]{\text{option}(#1)}

\newcommand{\seq}[1]{\langle{#1}\rangle}

\newcommand{\caseof}[5]{\mathsf{case} ~ #1 ~\mathsf{of} ~ \left(#2.#3 \| #4.#5\right)}

\newcommand{\shift}[2]{\mathcal{S}#1.#2}
\newcommand{\reset}[1]{\# #1}
\newcommand{\inl}[1]{\iota_1#1}
\newcommand{\inr}[1]{\iota_2#1}

\newcommand{\reczero}{\text{R}}
\newcommand{\ifelse}[3]{\mathsf{if} ~ #1 ~ \mathsf{then} ~ #2 ~ \mathsf{else} ~ #3}

\DeclareMathOperator{\len}{len}

\newcommand{\boole}{{\mathsf{bool}}}
\newcommand{\destas}[4]{\mathsf{dest} ~ #1 ~\mathsf{as} ~ (#2.#3) ~\mathsf{in}~ #4}
\newcommand{\destasNewLine}[4]{\mathsf{dest} ~ #1 ~\\~~~~\mathsf{as} ~ (#2.#3) ~\mathsf{in}~ #4}

\newcommand{\fst}[1]{\pi_1#1}
\newcommand{\snd}[1]{\pi_2#1}

\theoremstyle{definition}
\newtheorem{axiom}{Axiom}

\newcommand{\dnsv}{DNS$^{\text{V}}$}
\newcommand{\dnsvs}{DNS$^{\text{V}}_S$}
\newcommand{\AC}{AC$^{0,0}$}
\newcommand{\ACnatbool}{AC!$^{0,\bool}$}

\newcommand{\axc}[1]{\AxiomC{#1}}
\newcommand{\uic}[2]{\RightLabel{\small{#2}}\UnaryInfC{#1}}
\newcommand{\bic}[2]{\RightLabel{\small{#2}}\BinaryInfC{#1}}
\newcommand{\tic}[2]{\RightLabel{\small{#2}}\TrinaryInfC{#1}}

\newcommand{\mqcplusS}[1]{\mathrm{MQC}_+(#1)}
\newcommand{\haomega}{HA$^\omega$}
\newcommand{\haomegaplusS}[1]{\mathrm{HA}^\omega_+(#1)}

\DeclareMathOperator{\hd}{head}
\DeclareMathOperator{\tail}{tail}

\newcommand{\OIC}{OI-$\bool$}
\newcommand{\OIX}{OI-$X$}

\title{A Direct Version of Veldman's Proof of Open Induction on Cantor Space via Delimited Control Operators\footnote{D. Ilik's work is covered by a Kurt G\"odel Research Prize Fellowship 2011.}\footnote{K. Nakata acknowledges the ERDF funded EXCS project, the Estonian Ministry of Education and Research research theme no. 0140007s12, and the Estonian Science Foundation grant no. 9398.}}
\titlerunning{Open Induction via Delimited Control Operators}

\author[1]{Danko Ilik}
\author[2]{Keiko Nakata}

\affil[1]{Research Center for Computer Science and Information Technologies\\
  Macedonian Academy of Sciences and Arts\\
  Skopje, Macedonia\\
  \texttt{danko.ilik@gmail.com}}

\affil[2]{Institute of Cybernetics\\
  Tallinn University of Technology\\
  Tallinn, Estonia\\
  \texttt{keiko@cs.ioc.ee}}

\authorrunning{D. Ilik and K. Nakata} 

\Copyright{D. Ilik and K. Nakata}



\begin{document}

\maketitle

\begin{abstract}
First, we reconstruct Wim Veldman's result that Open Induction on Cantor space can be derived from Double-negation Shift and Markov's Principle. In doing this, we notice that one has to use a countable choice axiom in the proof and that 
Markov's Principle is replaceable by slightly strengthening the Double-negation Shift schema. We show that this strengthened version of Double-negation Shift can nonetheless be derived in a constructive intermediate logic based on delimited control operators, extended with axioms for higher-type Heyting Arithmetic. We formalize the argument and thus obtain a proof term that directly derives Open Induction on Cantor space by the shift and reset delimited control operators of Danvy and Filinski.
\end{abstract}

\section{Introduction}

Let $X$ be a set with an equality relation $=_X$ and a binary relation $<_X$. 
We denote by $X^\omega$ and $X^*$ the set of
infinite sequences, or \emph{streams}, over $X$
and the set of finite sequences over $X$, respectively. 
Let elements of $X^\omega$ be denoted by Greek letters $\alpha, \beta, \gamma$, let natural numbers be denoted by $n, k, l, m$, and let $\overline\alpha n$ denote the finite sequence $\langle\alpha(0),\alpha(1), \ldots, \alpha(n-1)\rangle$, i.e., the initial segment of length $n$ of the sequence $\alpha$.

The lexicographic extension $<_{X^\omega}$ of $<_X$ is a binary relation on streams, defined by
\[
\alpha<_{X^\omega}\beta \text{ iff } \exists n(\overline\alpha n =_{X^*} \overline\beta n 
\wedge \alpha(n) <_X \beta(n)),
\]
where $=_{X^*}$ denotes the equality relation induced from $=_X$ by element-wise
comparison, i.e., $p =_{X^*} q$ iff $p$ and $q$ are of the same length and
element-wise equal with respect to $=_X$.

A non-empty subset $U$ of $X^\omega$ is called \emph{open} if there is
an enumeration $\pi : \bn\to X^*$ which can approximate $U$, in the
sense that membership in $U$ can be defined\footnote{For simplicity,
  we exclude the possibility of $U=\emptyset$, so that we may take
  \emph{total} enumerations $\pi$, rather than partial enumerations,
  sending $\bn$ to $\option{X^*}$.} by
\[
\alpha\in U \text{ iff } \exists n\exists k(\overline\alpha n =_{X^*} \pi(k)).
\]

The \emph{Principle of Open Induction on $X^\omega$} (equipped with $<_X$ and $=_X$) is the following statement, for $U$ open: 
\[\tag{\OIX}
  \forall\alpha\left(\forall\beta<_{X^\omega}\alpha\left(\beta\in U\right)\to\alpha\in U\right) \to \forall\alpha(\alpha\in U).
\]

One immediately sees that \OIX\ has the form of a well-founded induction principle. However, one should note that, even for the simple choice of $X=\{0,1\}$ equipped with the usual decidable order and equality relation, an open set $U$ is generally uncountable, and the lexicographic ordering $<_{X^\omega}$ is not well-founded!

The utility of this principle has been recognized by Raoult \cite{Raoult1988} who gave, using \OIX, a new version of Nash-Williams' proof of Kruskal's theorem that does not explicitly use the Axiom of Dependent Choice\footnote{Raoult proves
\OIX\ using Zorn's Lemma.}. 

\OIX\ was introduced in the context of Constructive Mathematics by Coquand \cite{Coquand1992}. He proved \OIX\ by relativized Bar Induction, and also first considered separately the version for $X^\omega$ being the Cantor space \cite{Coquand1997}.
 
Berger \cite{Berger2004} showed that \OIX\ in higher-type Arithmetic, where $X$ can be any type $\rho$, is classically equivalent to the Axiom of Dependent Choice (DC) for the type $\rho$. He also gave
a modified realizability interpretation of \OIX\ by a schema of Open Recursion, and showed that, unlike DC, \OIX\ is closed under double-negation- and A-translation -- this means that there is a simple way
to extract open-recursive programs from classical proofs
of $\Pi^0_2$-statements that use DC or \OIX.

In the context of Constructive Reverse Mathematics, in a series of lectures \cite{Veldman2010}, Veldman showed that Open Induction for Cantor space is equivalent to Double-negation Shift,
\begin{equation}\tag{DNS}\label{equation:dns} 
\forall n \neg\neg A(n) \to \neg\neg\forall n A(n) \quad\text{(for any formula $A(n)$)},
\end{equation}
in presence of Markov's Principle,
\begin{equation}\tag{MP}\label{equation:mp} \neg\neg\exists n A_0(n) \to \exists n A_0(n) \quad\text{(for a decidable $A_0(n)$)}.
\end{equation}

Given that it is possible to obtain proofs for both MP \cite{Herbelin2010} and DNS \cite{Ilik2010} using constructive logical systems based on delimited control operators, it is a natural next step to attempt to provide a direct constructive proof of OI for Cantor space based on delimited control operators. This is what we do in this paper.

The remainder of the paper is organized as follows. In Section~\ref{sec:dns2oi}, we reconstruct in detail Veldman's argument that proves OI on Cantor space from DNS and MP via the principle EnDec. In Section~\ref{sec:shift2endec}, we recall the logical system $\mqcplusS{S}$ from \cite{Ilik2010} that is able to prove a strengthened version DNS$_S$ of DNS using delimited control operators. DNS$_S$ allows us to prove (a minimal logic version of) EnDec without explicitly using MP. In Section~\ref{sec:proofterm}, we give a formalized proof term for OI on Cantor space in a variant of HA$^\omega$ based on the logical system $\mqcplusS{S}$. In the concluding Section~\ref{sec:conclusion}, we explain the current limitation of our approach for extracting proofs from programs and we mention directly related works.

\section{From DNS and MP  to Open Induction for Cantor Space}\label{sec:dns2oi}

We will consider the case $X = \bool$, where $\bool = \{0,1\}$ with
$0 <_{\bool} 1$ and $0~=_{\bool}~0$, 
$1 =_{\bool} 1$, that is, Open Induction \emph{on Cantor space}, \OIC. We will show that \OIC\ is provable from DNS, MP, and \ACnatbool, where 
\[\tag{\ACnatbool}
\forall x^\bn\exists! y^\bool A(x,y) \to \exists f^{\bn\to\bool}\forall x^\bn A(x,f(x))
\]
is a restriction of the Axiom of Unique Countable Choice (also known as Countable Comprehension). All the arguments of this section take place in plain intuitionistic logic; if a principle that is not intuitionistically derivable is used, that is explicitly noted.

In addition to the already introduced notational conventions, let
$p,q,r,s$ denote finite binary sequences (bit-strings), $\bool^*$, and
let $p*q$ denote the concatenation of $p$ and $q$. 
For a natural number $k$, $\bool^k$ denotes the set of bit-strings
of length $k$. Concrete 
bit-strings are constructed using the notation $\langle\cdot\rangle$,
e.g. $\seq{}$ denotes an empty sequence, $\seq{0}$ the bit-string of
length 1 that contains a 0, $\seq{1,1,1,1}$ the bit-string that
contains four 1's, etc.  Thus $p*\seq{0}$ means that a zero bit is
appended at the end of $p$. The function $\len(p)$ computes the
length of $p$. Analogously to the initial segment function
$\overline\alpha n$ on infinite sequences, we denote by $\overline{p}
n$ the initial segment function on finite sequences, with default
value $\overline{p}n := p$ when $n>\len(p)$. Instead of writing
$<_{\bool^\omega}$ and $=_{\bool^*}$, we simply write $<$ and $=$. We
abbreviate $(S_1 \to S_2) \wedge (S_2 \to S_1)$ to $(S_1
\leftrightarrow S_2)$. We may write $n \not\in A$ to mean $\neg(n \in
A)$.

By a \emph{$\Sigma$-formula}, we mean a formula built only from 
existential quantifiers (over the set $\bn$),
disjunction,
conjunction,
and the equality symbol ``$=$'' for $\bn$.
This definition is equivalent to the
usual definition of $\Sigma^0_1$-formula if the language has
all the primitive recursive symbols, as is the case for the system from Section~\ref{sec:proofterm}.

We say that a set $B\subseteq\bn$ is \emph{enumerable} when 
the membership in $B$ is a $\Sigma$-formula, i.e.,
$n \in B$ is defined as $S(n)$ for a $\Sigma$-formula $S$.
Equivalently\footnote{``Equivalent'' in the system from Section~\ref{sec:proofterm}.}, $B$ is enumerable when $B$ is given by a function 
$f:\bn \to \bn$ such that $n \in B$ is a notation for 
$\exists m(f(m) = n+1)$. 
A set $B\subseteq\bn$ is \emph{decidable} when we have that $\forall n (n \in B \vee n \not\in B)$\footnote{In some literature, our ``decidable'' is called 
``detachable''.}. 

\medskip

Veldman introduced the following principle.
\begin{axiom}[EnDec]\label{axiom:endec} Assume $B\subseteq \bn$ is enumerable. Let, for any decidable $C\subseteq B$, we have that, if $\exists m (m\not\in C)$, then $\exists m (m\not\in C \wedge m\in B)$. Then $\bn\subseteq B$ (and hence $B$ is decidable).
\end{axiom}
Note that EnDec holds classically, since classically any $B$ is decidable, so we may set $C:=B$ to obtain $\bn\subseteq B$. Our interest in EnDec here is because it is a stepping stone to proving \OIC.

\begin{theorem}\label{thm:endec2oi}
Assuming \ACnatbool, EnDec implies Open Induction on Cantor space.
\end{theorem}
\begin{proof}
Let $A$ be a non-empty open subset of Cantor space\footnote{The progressiveness on Cantor space in fact ensures that $A$ is non-empty.} i.e., there exists $\pi : \bn\to \bool^*$ such that ``$\alpha\in A$'' is a notation for $\exists l,m(\overline\alpha l = \pi(m))$.
Let also $A$ be \emph{progressive}, that is,
\[
\forall \alpha (\forall\beta<\alpha(\beta\in A)\to \alpha\in A).
\]
We want to show that $\forall\alpha(\alpha\in A)$. Define $B\subseteq \bool^*$
as 
\[
p\in B \text{ iff } \exists k \forall q\in\bool^k\exists l,m (\overline{p*q}\, l = \pi(m))
\]
such that $p$ is in $B$ if $p$ is ``uniformly barred'' by $\pi$. That
is, $p\in B$ if there exists $k$ such that any extension of $p$ by a
finite bit-string of length $k$ is covered by $\pi(m)$ for some
$m$\footnote{A bit-string $p$ is \emph{covered} by $q$ if, as a
  bit-string, $q$ is a prefix of $p$, or the open set given by $p$ is
  covered by the open set given by $q$.}.

It suffices to show $\seq{}\in B$ for the empty bit-string $\seq{}$,
since we then know that $\pi$ covers the entire Cantor space. 
We show that $B$ is actually equal to $\bool^*$, using EnDec. Notice that $\bool^*$ is bijective to $\bn$ by primitive recursive functions and $B$ is enumerable\footnote{$B$ is enumerable because it is defined by a $\Sigma$-formula: the bounded universal quantifier ``$\forall q\in\bool^k$'' does not pose a problem, since it could be interpreted as a bounded minimization operator, for example like in \S 3.5 of \cite{Kohlenbach}.},
hence we may transport EnDec from $\bn$ to $\bool^*$.
It is left to show that, for any decidable subset $C\subseteq B$, if $\exists q(q\not\in C)$, then $\exists r(r\not\in C \wedge r\in B)$.

Suppose that such $C$ and $q$ are given. If $\seq{} \in C \subseteq B$, 
then we have that $q \in B$. So we are done. We assume $\seq{} \not\in C$.
Since $C$ is decidable, we can construct
$\alpha$, using \ACnatbool, such that
\[
\alpha(n) := \left\{ 
\begin{array}{ll}
  0 & \text{, if } \overline\alpha n * \seq{0} \not\in C\\
  1 & \text{, if } \overline\alpha n * \seq{0} \in C \text{ and } \overline\alpha n * \seq{1} \not\in C\\
  0 & \text{, if } \overline\alpha n * \seq{0} \in C \text{ and } \overline\alpha n * \seq{1} \in C
\end{array}
\right.
\]
The sequence $\alpha$ tries to stay outside of $C$ for as long as
possible and tries to be minimal. It first tries to ``turn left''
(value $0$). If it was not possible, i.e., $\overline\alpha n * \seq{0} \in
C$, then it tries to ``turn right'' (value $1$). If neither was possible,
then it defaults to ``turning left''.  One may notice that if $\alpha$ fails to stay
outside of $C$ at $n+1$, i.e., $\overline\alpha n * \seq{0} \in C$ and
$\overline\alpha n * \seq{1} \in C$, then we have 
$\overline{\alpha}n \in B$. This fact, a manifestation of the compactness of Cantor space, will be used later in the proof. 

Now, we can find a prefix of $\alpha$ that is in $B$ but not in $C$,
by following $\alpha$ up to the first point where it enters $B$.  Let
us first prove that $\alpha$ is in $A$, which guarantees that $\alpha$
has a prefix in $B$, hence that $\alpha$ will enter $B$.  We use
progressiveness of $A$. Let $\beta<\alpha$ i.e., $\exists n
(\overline\beta n = \overline\alpha n \wedge \beta(n)=0 \wedge
\alpha(n)=1)$. We have to show $\beta \in A$.  By construction of
$\alpha$, $\alpha(n)=1$ is only possible if $\overline\alpha
n*\seq{0} \in C$ and $\overline\alpha n*\seq{1} \not\in C$.
Noticing that $\overline\beta(n+1) = \overline\beta n*\seq{0} =
\overline\alpha n*\seq{0}$, this yields 
$\overline\beta(n+1)\in C\subseteq B$. We conclude that $\beta\in A$,
which was to be shown.

From $\alpha\in A$, we obtain $l, m$ such that $\overline\alpha l = \pi(m)$.
We finish the proof by proving the following more general statement by induction 
\[
\forall n\le l\left(\overline\alpha(l-n)\not\in C \to \exists l'(\overline\alpha l'\not\in C \wedge \overline\alpha l' \in B)\right).
\]
Indeed, since we have $\seq{} \not\in C$, by instantiating the above statement
with $n := l$, we obtain $p$ such that $p \not\in C$ and $p \in B$.

In the base case, $n=0$, we have that $\overline\alpha l\not\in C$
by the hypothesis and that $\overline\alpha l\in B$ (from $\alpha \in A$); 
so we set $l':=l$.
In the induction case for $n+1$ we consider three possibilities:
  \begin{enumerate}
  \item if $\overline\alpha(l-(n+1))*\seq{0}\not\in C$, then 
$\overline\alpha(l-n) = \overline\alpha(l-(n+1)+1) = 
\overline\alpha(l-(n+1))*\seq{0} \not \in C$ and we close the case by induction hypothesis;
  \item similarly, if $\overline\alpha(l-(n+1))*\seq{0}\in C$ and $\overline\alpha(l-(n+1))*\seq{1}\not\in C$, then $\overline\alpha(l-n) = \overline\alpha(l-(n+1)+1) = \overline\alpha(l-(n+1))*\seq{1} \not\in C$, and we close the case
by induction hypothesis;
  \item if $\overline\alpha(l-(n+1))*\seq{0}\in C$ and $\overline\alpha(l-(n+1))*\seq{1}\in C$, then we get that $\overline\alpha(l-(n+1))\in B$ as we noted earlier. Recalling that we also have $\overline\alpha(l-(n+1))\notin C$ by
hypothesis, we can set $l' := l-(n+1)$. 
  \end{enumerate}
The first two cases could be merged into one, verifying only whether $\overline\alpha(l-(n+1)+1)\not\in C$.
\end{proof}

\begin{remark}\label{remark:ac} 
In the previous proof, we used \ACnatbool\ when constructing the
sequence $\alpha$ by course-of-values recursion using the choice
function extracted from the decidability of $C$. Since the principle
EnDec is classically valid, not using a choice axiom would mean that
one can reduce \OIC\ (and, using Berger's results \cite{Berger2004},
also Dependent Choice for $\bool$) to plain classical logic
without choice\footnote{Classically \ACnatbool\ is equivalent to Dependent 
Choice for $\bool$ (in Berger's formulation), hence that we only use \ACnatbool\ is not a 
concern.}.
\end{remark}

We now consider the principle of Double-negation Shift (DNS), which is independently important because it allows to interpret the double-negation translation of the Axiom of Countable Choice \cite{Spector}. Following Veldman, we find it useful to consider the following variant of DNS.
\begin{axiom}[\dnsv]\label{axiom:dnsv} $\neg\neg \forall n(A(n) \vee \neg A(n))$, for any formula $A(n)$.
\end{axiom}
\begin{remark} The proof of equivalence between DNS and \dnsv\ is analogous to the proof of equivalence between the law of double-negation elimination (DNE) and the law of excluded middle (EM). In minimal logic, which is intuitionistic logic without the rule of $\bot$-elimination (\textit{ex falso quodlibet}), EM is weaker than DNE \cite{AriolaH2003}. We expect a similar result for DNS, i.e., that \dnsv\ is weaker than DNS in minimal logic.
\end{remark}

When quantifier-free formulas and decidable formulas coincide, as in Arithmetic, we may state Markov's Principle using $\Sigma$-formulas.
\begin{axiom}[MP]\label{axiom:mp} For any $\Sigma$-formula $S$, we have that $\neg\neg S \to S$.
\end{axiom}

We can now prove EnDec from \dnsv and MP. 

\begin{theorem}\label{thm:dnsmp2endec} \dnsv  and 
MP  together imply EnDec.
\end{theorem}
\begin{proof}
Let the premises of EnDec hold. Given $n\in \bn$, we have to prove
$n\in B$, which is a $\Sigma$-formula. 
We are entitled to apply MP. Now, we have to show that $\neg\neg(n\in B)$.
Suppose $\neg(n\in B)$. Thanks to \dnsv, it suffices to prove $\bot$
assuming moreover that $B$ is decidable, i.e., 
$\forall n (n\in B \vee \neg(n\in B))$.
We use the premise of EnDec by taking
$C := B$ and recalling that we have $\neg(n\in B)$.
This gives us $\exists m (m\in B \wedge \neg(m\in B))$, from which we derive $\bot$.
\end{proof}

\section{A Constructive Logic Proving EnDec}\label{sec:shift2endec}

In this section, we recall the logical system $\mqcplusS{S}$ from \cite{Ilik2010}, 
and show that EnDec is provable in $\mqcplusS{S}$ 
(with a suitably instantiated parameter $S$), 
without an explicit use of MP, thanks to the slightly stronger form of DNS 
that $\mqcplusS{S}$ proves.

$\mqcplusS{S}$ is a pure predicate logic system, 
parameterized over a closed $\Sigma$-formula $S$, 
that, in addition to the usual rules of minimal intuitionistic predicate logic, adds two rules for proving the $\Sigma$-formula $S$
\footnote{In the context of $\mqcplusS{S}$, $\Sigma$-formulas coincide with formulas without $\forall$ and $\to$.}. 
The rule ``reset'',
    \begin{prooftree}
      \axc{$\Gamma\vdash_S S$}
      \uic{$\Gamma\vdash_\diamond S$}{$\#$ (``reset''),}
    \end{prooftree}
sets a marker (under the turnstile) meaning that one wants to prove $S$. Once the marker is set, one can use the ``shift'' rule,
    \begin{prooftree}
      \axc{$\Gamma, A\Rightarrow S\vdash_S S$}
      \uic{$\Gamma\vdash_S A$}{$\mathcal{S}$ (``shift''),}
    \end{prooftree}
to prove by a principle related to double-negation elimination from classical logic. The idea is to internalize in the formal system the fact, known from Friedman-Dragalin's A-translation, that a classical proof of a $\Sigma^0_1$-formula can be translated to an intuitionistic proof of the same formula, showing that classical proofs of such formulas are in fact constructive. The first system built around this internalization idea was Herbelin's \cite{Herbelin2010} with the power to derive Markov's Principle. It satisfies, like $\mqcplusS{S}$, the disjunction and existence properties, characteristic of plain intuitionistic logic.

The names ``shift'' and ``reset'' come from the computational intention behind the normalization of these proof rules, Danvy and Filinski's  delimited control operators \cite{DanvyF1989,DanvyF1990,DanvyF1992}. These operators were developed in the theory of programming languages with the aim of enabling to write continuation-passing style (CPS) programs in so-called \emph{direct style}. Since CPS transformations are known to be one and the same thing as double-negation translations \cite{MurthyThesis}, one can think of shift/reset in Logic as enabling to prove \emph{directly} theorems whose double-negation translation is intuitionistically provable. In order for this facility to remain constructive, we allow its use only for proving $\Sigma$-formulas.

\begin{table}
  \centering
  \begin{tabular}{ m{5cm} m{6cm} }
    \multicolumn{2}{ m{11cm} }{
      \begin{prooftree}
        \axc{$(a:A) \in \Gamma$}
        \uic{$\Gamma\vdash_\diamond a:A$}{\textsc{Ax}}
      \end{prooftree}
    }
    \\
    \begin{prooftree}
      \axc{$\Gamma\vdash_\diamond p:A_1$}
      \axc{$\Gamma\vdash_\diamond q:A_2$}
      \bic{$\Gamma\vdash_\diamond (p,q):A_1\wedge A_2$}{$\wedge_I$}
    \end{prooftree}
    &
    \begin{prooftree}
      \axc{$\Gamma\vdash_\diamond p:A_1\wedge A_2$}
      \uic{$\Gamma\vdash_\diamond \pi_i\, p:A_i$}{$\wedge^i_E$}
    \end{prooftree}
    \\
    \begin{prooftree}
      \axc{$\Gamma\vdash_\diamond p:A_i$}
      \uic{$\Gamma\vdash_\diamond \iota_i\, p:A_1\vee A_2$}{$\vee^i_I$}
    \end{prooftree}
    & \\
    \multicolumn{2}{ m{11cm} }{
      \begin{prooftree}
        \axc{$\Gamma\vdash_\diamond p:A_1\vee A_2$}
        \axc{$\Gamma, a_1:A_1\vdash_\diamond q_1:C$}
        \axc{$\Gamma, a_2:A_2\vdash_\diamond q_2:C$}
        \tic{$\Gamma\vdash_\diamond \caseof{p}{a_1}{q_1}{a_2}{q_2}:C$}{$\vee_E$}
      \end{prooftree}
    }
    \\
    \begin{prooftree}
      \axc{$\Gamma, a:A_1\vdash_\diamond p:A_2$}
      \uic{$\Gamma\vdash_\diamond \lambda a.p:A_1\to A_2$}{$\to_I$}
    \end{prooftree}
    &
    \begin{prooftree}
      \axc{$\Gamma\vdash_\diamond p:A_1\to A_2$}
      \axc{$\Gamma\vdash_\diamond q:A_1$}
      \bic{$\Gamma\vdash_\diamond p\, q:A_2$}{$\to_E$}
    \end{prooftree}
    \\
    \begin{prooftree}
      \axc{$\Gamma\vdash_\diamond p:A(x)$}
      \axc{$x~\text{fresh}$}
      \bic{$\Gamma\vdash_\diamond \tilde\lambda x.p:\forall x  A(x)$}{$\forall_I$}
    \end{prooftree}
    &
    \begin{prooftree}
      \axc{$\Gamma\vdash_\diamond p:\forall x A(x)$}
      \uic{$\Gamma\vdash_\diamond p\, t:A(t)$}{$\forall_E$}
    \end{prooftree}
    \\
    \begin{prooftree}
      \axc{$\Gamma\vdash_\diamond p:A(t)$}
      \uic{$\Gamma\vdash_\diamond (t,p):\exists x.A(x)$}{$\exists_I$}
    \end{prooftree}
    & \\
    \multicolumn{2}{ m{11cm} }{
      \begin{prooftree}
        \axc{$\Gamma\vdash_\diamond p:\exists x.A(x)$}
        \axc{$\Gamma, a:A(x)\vdash_\diamond q:C$}
        \axc{$x~\text{fresh}$}
        \tic{$\Gamma\vdash_\diamond \destas{p}{x}{a}{q}:C$}{$\exists_E$}
      \end{prooftree}
    }
    \\
    \begin{prooftree}
      \axc{$\Gamma\vdash_S p:S$}
      \uic{$\Gamma\vdash_\diamond \reset{p}:S$}{$\#$ (``reset'')}
    \end{prooftree}
    &
    \begin{prooftree}
      \axc{$\Gamma, k:A\to S\vdash_S p:S$}
      \uic{$\Gamma\vdash_S \shift{k}{p}:A$}{$\mathcal{S}$ (``shift'')}
    \end{prooftree}\\
    \multicolumn{2}{ m{11cm} }{
    }\\
    ~ & ~
    \\
  \end{tabular}
  
  \caption[\mqcplus with proof terms]{Natural deduction system for $\mqcplusS{S}$, parameterized over a closed $\Sigma$-formula $S$, with proof terms annotating the rules}
  \label{table:mqcplus}
\end{table}
The natural deduction system for $\mqcplusS{S}$ 
is given in Table~\ref{table:mqcplus} with proof term annotations. The diamond in the subscript of $\vdash$ is a wild-card: $\vdash_\diamond$ denotes either $\vdash$ or $\vdash_S$, where in the latter the subscript $S$ is the same formula as the parameter $S$. We mark $\vdash$ with
the parameter to record that a reset has been set. 
The rules should be read bottom-up, so that the marker is propagated from below to above the line. The usual intuitionistic rules neither ``read'' nor ``write'' this marker, hence $\diamond$ denotes the same below and above the line.
 The reset rule is the one that sets the marker (if it is not already set). 
If the marker has been already set, then the marker is simply kept. 
This kind of use of reset would have no logical purpose, 
but it would affect the course of normalization, hence the computational behavior of 
the proof term. 
The rule shift can only be applied when the marker is set, hence
it is assured that we are ultimately proving the $\Sigma$-formula $S$. 

\medskip

The following theorem shows a utility of proving with shift and reset.
\begin{theorem} Let $S$ be a closed $\Sigma$-formula and $A(x)$ an arbitrary formula. The following version of \dnsv,
\begin{equation}\tag{\dnsvs}
\bigg(\Big(\forall x \big(A(x) \vee \left(A(x) \to S\right)\big)\Big)\to S\bigg)\to S,
\end{equation}
is provable in $\mqcplusS{S}$.
\end{theorem}
\begin{proof} Using the proof term $\lambda h. \reset{h \bigg(\tilde\lambda x. \shift{k}{k\Big(\inr{\big(\lambda a. k (\inl{a})\big)}\Big)}\bigg)}$.
\end{proof}
\dnsvs\ is a version of \dnsv, in which $\bot$ is generalized
to a closed $\Sigma$-formula $S$. \dnsvs\ already has some form of MP built in, as can be seen from the proof of Theorem~\ref{thm:dns2endec} below.

We now state a version of EnDec which is suitable for use in minimal logic, 
where $\bot$-elimination is absent. 
\begin{axiom}[A minimal-logic version of Axiom~\ref{axiom:endec}]\label{axiom:endecmin}
Assume that $B\subseteq \bn$ is enumerable and $n\in\bn$. Let, for any $s\in\bn$ and any $C\subseteq B$, such that
\[
\forall x \left(x\in C \vee \left(x\in C \to s\in B\right)\right),
\]
we have that, if \[\exists m (m\in C\to s\in B),\] then \[\exists m ((m\in C\to s\in B) \wedge m\in B).\] Then, $n\in B$.
\end{axiom}

The following result is the minimal-logic analogue of Theorem~\ref{thm:dnsmp2endec}, showing that an instance of Axiom~\ref{axiom:endecmin} is derivable in $\mqcplusS{S}$.
\begin{theorem}\label{thm:dns2endec} 
Assume that $B\subseteq \bn$ is enumerable and $n\in\bn$. 
The instance of Axiom~\ref{axiom:endecmin} with conclusion $n\in B$ is derivable in the system $\mqcplusS{n\in B}$.
\end{theorem}
\begin{proof} Let the premises of Axiom~\ref{axiom:endecmin} hold. To show that $n\in B$, which is a $\Sigma$-formula, we use \dnsvs\ for $A(x):=x\in B$ and $S:=n\in B$. Now, given $\forall x (x\in B \vee (x\in B \to n\in B))$, we have to show $n\in B$. We use the premise of Axiom~\ref{axiom:endecmin} for $s:=n$ and $C:=B$, and, using the trivial proof of $\exists m (m\in B \to n\in B)$ for $m:=n$, the premise gives us a proof of $\exists m (m\in B \wedge (m\in B \to n\in B))$, from which we derive $n\in B$.
\end{proof}

\section{A Proof Term for Open Induction}\label{sec:proofterm}

In this section, we give a proof term for OI on Cantor space in the
system $\haomegaplusS{S}$ (by suitably instantiating the parameter $S$), 
which is the system of axioms \haomega\ (from
\S\S 1.6.15 of \cite{TroelstraBook}) and \ACnatbool\ added on top of
the predicate logic $\mqcplusS{S}$ --- 
the need of \ACnatbool\ is justified by 
Remark~\ref{remark:ac}. 
Basic ingredients to construct the
proof term are at hand: Theorem~\ref{thm:endec2oi} and 
Theorem~\ref{thm:dns2endec}. We are to interpret them in 
$\haomegaplusS{S}$ and combine the thus obtained proof terms for 
Theorem~\ref{thm:endec2oi} and Theorem~\ref{thm:dns2endec}.

\subsection{The system $\haomegaplusS{S}$}

Let $S$ be a closed $\Sigma$-formula.
First, we take a multi-sorted version of $\mqcplusS{S}$, that is, given different sorts (denoted by $\sigma,\rho,\tau, \delta$), the language is extended with individual variables (denoted by $x,y,z$) of any sort, and quantifiers for all sorts. We will not annotate quantifiers with their sorts, since those will be clear from the context; we may annotate variables by their sorts when we want to avoid ambiguity.

The sorts are built inductively, according to the following rules: there is a sort named 0; if $\rho$ and $\sigma$ are sorts, then there is a sort named $\rho\to\sigma$. The intended interpretation is that the sort 0 stands for $\bn$, the sort $0\to 0$ stands for functions $\bn\to\bn$, the sort $((0\to 0)\to 0)$ for functionals $(\bn\to\bn)\to\bn$, etc. We will employ the word `type' instead of sort, henceforth, and we abbreviate the type $0\to 0$ by $1$.

Now, we add to the language a binary predicate symbol $=$ for individual terms of type 0, intended to be interpreted as (the decidable) equality on $\bn$. We emphasize that we only have decidable equality. The individual terms will be built from the function symbols $0^0$ (zero), $(\cdot + 1)^1$ (successor), $\Pi^{\rho\to\tau\to \rho}$ and $\Sigma^{(\delta\to \rho\to \tau) \to (\delta\to\rho) \to \delta \to \tau}$ (combinators), and $\reczero^{0\to \rho\to(\rho\to 0\to \rho)\to \rho}$ (recursor of type $\rho$). There is also the function symbol of juxtaposition which is not explicitly denoted: for terms $t^{\sigma\to\tau}$ and $s^\sigma$, $t\, s$ is a term of type $\tau$.

The axioms defining these symbols are (the universal closures of each of):
\begin{align*}
  x = x, & & x = y \to y = x,  & & x = y \to y = z \to x = z, & & x = y \to x+1 = y+1,
\end{align*}
\begin{align*}
  x = y \to t[x/z] = t[y/z] & & \text{where $t[x/z]$ is the simultaneous}\\
                               & & \text{substitution of $x$ for $z$ in $t$}
\end{align*}
\begin{align*}
  t[\Pi x y / u] &= t[x / u]\\
  t[\Sigma x y z / u] &= t[x z (y z) / u]\\
  t[\reczero 0 y z / u] &= t[y / u]\\
  t[\reczero (x+1) y z / u] &= t[z (\reczero x y z) x / u]
\end{align*}
We also add the axiom schema of induction, for arbitrary formula $A(x)$, but only for variables $x$ of type $0$:
\begin{equation}\tag{IA}
  A(0) \to \forall x^0 (A(x) \to A(x+1)) \to \forall x^0 (A(x))
\end{equation}
Since ``$=$'' is the only predicate symbol, all atomic (prime) formulas are of form $t = s$. This allows us to show that $x = y \to A(x)\to A(y)$, by induction on the complexity of formula $A$. 

It is known that using the combinators one may define an individual term for lambda abstraction, denoted $\dot\lambda x. t$, of type $1$, which satisfies the usual $\beta$-reduction axiom, \[(\dot\lambda x^0. s^0) t^0 = s[t/x].\] Using this and the recursor $\reczero$, one can easily define all the usual primitive recursive functions.  Using the thus defined predecessor function, and the induction axiom, one can derive the remaining Peano axioms, $x+1 = y+1 \to x = y$, and $(x+1 = 0) \to 1= 0$, where we took $1= 0$ instead of $\bot$ because we are in minimal logic. In fact, in the presence of arithmetic, one can prove, again by induction, that the rule of $\bot$-elimination (with $\bot$ replaced by $1=0$) is derivable, although we will not need it.

Some notational conventions follow. We shall need to speak of bits, finite sequences of bits (bit-strings), and infinite sequences of bits (bit-streams). Bits and bit-strings can be encoded by natural numbers, but, instead of using the type 0 for terms of that kind, to be more pragmatic, we will write $\boole$ (intended to interpret $\bool$) and $\boole^*$ (intended to interpret $\bool^*$). 
Bitstreams are represented by terms of type $0\to 0$, but we will write $0\to\boole$ instead. 
We will need the operations for concatenation and initial segments of both bit-strings and bit-streams, that we already introduced. In addition, the operator $\hd(p)$ returns the first bit of $p$, while $\tail(p)$ returns the string that follows the first bit of $p$.  
Although $p$ is not a function, we will use the notation $p(n)$ to extract the $(n+1)$-th bit of $p$\footnote{$\hd{p}$ (resp. $p(n)$) returns an arbitrary default value when $p$ is an empty sequence (resp. $\len(p) < n+1$). However, we will use these operations only in a well-defined way.}. We will also use the fact that one can define by primitive recursion a term $\ifelse{\cdots}{\cdots}{\cdots}$ of type $\boole\to \boole\to \boole \to \boole$, such that the following equations hold:
\begin{align*}
  \ifelse{0}{y}{z} &= z\\
  \ifelse{x+1}{y}{z} &= y
\end{align*}
We will also need the usual operation $\min : 0\to 0\to 0$ on numbers. All the mentioned operations can be defined by a restricted amount of primitive recursion at higher types, level 3 of the Grzegorcyk hierarchy would suffice. Hence we could work in a corresponding subsystem of \haomega, like for example G$_3$A$^\omega_i$ from \S 3.5 of \cite{Kohlenbach}.

\medskip

Finally, we shall also need the following choice axiom, a restriction of the usual Axiom of Countable Choice (\AC): 
\begin{equation}\tag{\ACnatbool}
  \forall x^0\exists! y^\boole A(x,y) \to \exists \phi^{0\to\boole}\forall x^0 A(x,\phi\, x)
\end{equation}
Neither \AC\ nor \ACnatbool\ is provable in \haomega. For arithmetical formulas, \AC\ (and hence \ACnatbool) is an admissible rule for \haomega~\cite{Beeson1979}.

\subsection{Proof term for \OIC}

We now formalize the concepts involved in the proof of \OIC.
An open set $A$ in Cantor space 
is given, as a parameter to the logical system, 
by a term $\pi$ of type $0\to \boole^*$, an enumeration of basic opens. 
Each bit-string $\pi(n)$ is a basic open and the union of them makes
$A$. Membership in $A$, $\alpha\in A$, means that $\alpha$ is covered by some basic open from the
enumeration. Formally, we define
\[
\alpha\in A \text{ iff } \exists l^0\exists m^0 (\overline\alpha\, l = \pi(m)),
\]
and we see that membership in $A$ is a closed $\Sigma$-formula. 
(Recall that $\pi$ is a parameter of the logical system.)
The relation $<$ on bit-streams is formalized as
\[
\beta<\alpha \text{ iff } \exists n^0 \left(\overline\beta n =\overline\alpha n \wedge (\beta(n) = 0 \wedge \alpha(n) = 1)\right).
\]

We use an instance of Axiom~\ref{axiom:endecmin} for the enumerable set
$B$ given by a $\Sigma$-formula $B(x)$, to be defined below,
and $n$ given by the natural number encoding an empty sequence. 
We define 
\[
B(x) := \exists k^0\forall q^{\boole^k}\exists l^0\exists m^0 (\overline{x*q}\, l = \pi(m)),
\]
where $\forall q^{\boole^k}$ denotes a \emph{bounded} universal quantification over bit-strings of length $k$. 
Boun\-ded quantification can be encoded away using primitive recursive symbols, 
hence $B(x)$ is still a $\Sigma$-formula. 
We define $p \in B$ by $B(p)$. We have that, for any $\alpha$, $\exists n(\overline\alpha n \in B)$ iff $\alpha\in A$. We instantiate the parameter $S$ of $\haomegaplusS{S}$ by $\seq{} \in B$.

Next, we give an interpretation of the instance
of Axiom~\ref{axiom:endecmin} in $\haomegaplusS{\seq{} \in B}$. 
We cannot literally formalize Axiom~\ref{axiom:endecmin} in 
$\haomegaplusS{S}$,
since $\haomegaplusS{S}$ does not have higher-order quantification (but only 
quantification over higher types), hence we cannot quantify over
subsets. We therefore ``interpret'' (the instance of)
Axiom~\ref{axiom:endecmin}:
 \begin{multline*}
 \forall s^{\boole^*}\left(\forall \chi_C^{\boole^*\to\boole}\left(\forall x^{\boole^*}(\chi_C(x)=1\to B(x)) \to 
 \right.\right.\\
\left.\left.\exists q^{\boole^*}(\chi_C(q)=1 \to B(s)) \to  \right.\right.
\left.\left.\exists r^{\boole^*}\left((\chi_C(r)=1 \to B(s)) \wedge B(r)\right)
  \right)\right)
 \to B(\seq{}).
 \end{multline*}
The enumerable set $B$ is represented by the $\Sigma$-formula $B(x)$, 
the decidable subset $C$ by a characteristic function $\chi_C^{\boole^*\to\boole}$,
replacing the premise
$\forall x \left(x\in C \vee \left(x\in C \to s\in B\right)\right)$. 
The characteristic function should intuitively read as 
$\chi_C(p)=1$ iff ``$p\in C$'', but we take $B(s)$ for $\bot$.

The proof term for \OIC\ is shown in Figure~\ref{fig:OIprf}. 
We obtained it by formalizing the proofs of Theorems~\ref{thm:endec2oi} and~\ref{thm:dns2endec} in $\haomegaplusS{\seq{} \in B}$, and then by normalizing and (hand-)optimizing
the formalized proof term, to obtain a compact and direct program proving
\OIC. 

To ease the presentation, at certain places, we have put after a semicolon the type annotations for individual terms, and the formulas for proof terms. 
Some parts, being too long, have been put below the main proof term.
We suppress the use of equality axioms, to keep the proof term
simple without equality-rewriting terms.
It is known that equality proofs
have no computational content when extracting programs, as they are
realized by singleton data types.

\begin{figure}[t]

\[\begin{array}{ll}
1:~~~& 
\lambda h:\forall\alpha(\forall\beta<\alpha(\beta\in A)\to \alpha\in A).
\tilde\lambda \alpha'.\\
2:~~~&\destas{\\
3:~~~&~~\bigg(\reset{\destas{a_C(\tilde\lambda x. \shift{k}{k(\inr{(\lambda a. k(\inl{a}))})})}{\chi}{b}
{\\4:~~~&~~~~~~\destas{\Big(h \alpha \big(\tilde\lambda\beta. \lambda h':\beta<\alpha.\\ 
5:~~~&~~~~~~~~~~\destas{(h':\beta < \alpha)}{n}{h''}
{\\6:~~~&~~~~~~~~~~\destas{(a_1 (\snd{\snd{h''}}):\overline\beta(n+1)\in B)}{k}{h'''}
{\\7:~~~&~~~~~~~~~~\destas{(h''' (\seq{\beta(n+1)}*\cdots*\seq{\beta(n+k)}):\overline{\beta}(n+k+1) \in A)}}{j}{h^4}
{\\8:~~~&~~~~~~~~~~(\min(n+k+1,j), h^4)}}\big):\alpha\in A\Big)}{l}{c}
{\\9:~~~&~~~~~~\destas{(c:\exists m(\overline{\alpha}l = \pi(m))}{m}{d}
{\\10:~~~&~~~~~~a_I\, (\lambda h. h)\, a_3\, l\, (0, \tilde\lambda q.  (l,(m,d)))}}
}}:\seq{}\in B\bigg)}{k'}{h^5}{\\
11:~~~&\destas{(h^5\, (\overline{\alpha'}k'): \overline{\alpha'}k' \in A)}{j'}{h^6}{
\\12:~~~&(\min(k',j'), h^6)}}
\end{array}\]

$\begin{array}{l}
\alpha := \dot\lambda n. \\
~~\reczero(n+1,\seq{}, (\dot\lambda z. \dot\lambda n'. z*\seq{\ifelse{\chi(z*\seq{0})}{(\ifelse{\chi(z*\seq{1})}{0}{1})}{0}}))(n)
\\[1ex]
a_1: \alpha(n)=1 \to \overline\beta(n+1) \in B:= \lambda h. 
\mathsf{case}~a_B(\chi(\overline\beta(n+1)))~\mathsf{of}\\
~~\left(h_1.(\fst{(b(\overline\beta(n+1)))})\, h_1
\| h_2.(\fst{(b(\overline\beta(n+1)))})\, h_2\right)
\\[1ex]
a_3 := \tilde\lambda n. \lambda h_I:\overline{\alpha}n\in B \to \seq{}\in B.
\lambda h:\overline{\alpha}(n+1)\in B. \\ 
~~\mathsf{case}~a_B (\chi(\overline\alpha n * \seq{0}))~\mathsf{of}\, 
(h_1.
 (\snd{(b(\overline\alpha(n+1)))})\, h_1\, h\\
~~~~\|\, h_2.\caseof{(a_B (\chi(\overline\alpha n * \seq{1})))}
{h_{21}}{(\snd{(b(\overline\alpha(n+1)))})\, h_{21}\, h}
{h_{22}}{h_I\, a_4})
\\[1ex]
a_4: \overline{\alpha}n \in B := \\
~~\destasNewLine{((\fst{(b(\overline{\alpha}n*\seq{0}))})\, h_2:\overline{\alpha}n*\seq{0}\in B)}
{k_0}{f_0:\forall q:\boole^{k_0}.\exists l,m(\overline{\overline{\alpha}n*\seq{0}*q}\ l = \pi(m))}
{\\ ~~\destasNewLine{((\fst{(b(\overline{\alpha}n*\seq{1}))})\, h_{22}:\overline{\alpha}n*\seq{1}\in B)}
{k_1}{f_1:\forall q:\boole^{k_1}.\exists l,m(\overline{\overline{\alpha}n*\seq{1}*q}\ l = \pi(m)}
{\\ ~~(\min(k_0,k_1)+1, \lambda q:\boole^{\min(k_0,k_1)+1}.
  \ifelse{\hd(q)}{f_1 (\overline{\tail(q)}{k_1})}
  {f_0 (\overline{\tail(q)}{k_0})})}}
\end{array}$

\caption{Proof term for OI-$\bool$ of type $((\forall\alpha(\forall\beta<\alpha(\beta\in A)\to \alpha\in A))\to \forall\alpha'(\alpha' \in A))$ in 
$\haomegaplusS{\seq{}\in B}$.}\label{fig:OIprf}
\end{figure}

\smallskip

We now explain the behavior of the proof term. Given a proof $h$ that $A$ is progressive, it has to show that $\alpha' \in A$ for any $\alpha'$. As in the proof of 
Theorem~\ref{thm:endec2oi}, it proves $\seq{} \in B$ (lines 3-10), 
from which we obtain $k'$ such that 
$h^5:\forall q^{\boole^{k'}}\exists l^0\exists m^0 (\overline{q}\, l = \pi(m))$
(line 10).
Then $h^5 (\overline{\alpha'}k')$ gives us $j'$ such that $h^6:\exists m^0
(\overline{\overline{\alpha'}k'}j' = \pi(m))$ (line 11),
so that $(\min(k', j'), h^6)$ proves $\exists l^0 \exists m^0
(\overline{\alpha'}l = \pi(m))$ (line 12). 
(An explicit proof of the equality $\overline{\overline{\alpha'}k'}j'
= \overline{\alpha'}(\min(k', j'))$ 
would need an explicit definition of the $\min$ function and induction). 

To show $\seq{}\in B$, which is the parameter of the system, it applies a reset $\reset{}$ (line 3), and now it has to show the same formula, but classical logic in the form of the shift rule can be used. Indeed, the proof term $\tilde\lambda x. \shift{k}{k(\inr{(\lambda a. k(\inl{a}))})}$ proves the ``decidability'' of $B$: $\forall x^{\boole^*} (x\in B \vee (x\in B \to \seq{}\in B))$. Using the proof term $a_C$ for the formula
\begin{multline*}
\forall x^{\boole^*}(x\in B \vee (x\in B \to \seq{}\in B)) \to\\
\exists\chi^{\boole^* \to \boole}\forall x^{\boole^*}((\chi(x)=1 \rightarrow x\in B) \wedge 
(\chi(x)=0 \rightarrow (x\in B \to \seq{}\in B))),
\end{multline*}
we obtain from the decidability, a characteristic function $\chi^{\boole^*\to\boole}$ for $B$. The proof term $a_C$ is constructed by combining \ACnatbool\ together with a proof term that eliminates disjunction in presence of arithmetic\footnote{For the proof of this statement, $(A\vee B) \leftrightarrow \exists x ((x=1 \rightarrow A) \wedge (x=0 \rightarrow B))$, see for example \S\S 1.3.7 of \cite{TroelstraBook}.}. The proof term $b$ proves the characteristic property of $\chi$, namely,
$\forall x((\chi(x)=1 \rightarrow x\in B) \wedge 
(\chi(x)=0 \rightarrow (x\in B \to \seq{}\in B)))$.

Now, using this $\chi$, the bit-stream $\alpha$ that we saw in the proof of Theorem~\ref{thm:endec2oi} can be constructed using $\reczero$ and $\ifelse{\cdots}{\cdots}{\cdots}$
by (encoded) course-of-values recursion. 

Next one needs to show that $\alpha\in A$ (lines 4-8). 
One uses progressiveness $h$: from $\beta$ and a proof $h'$ of $\beta<\alpha$, one extracts $n$ and a proof $h''$ of
\[
\overline\beta n =\overline\alpha n \wedge (\beta(n) = 0 \wedge \alpha(n) = 1).
\]
Then, $\snd{\snd{h''}}$ shows $\alpha(n) = 1$, and it is for $a_1$ to show that $\overline\alpha n *\seq{0} = \overline\beta(n+1)$ is in $B$, 
which in turn shows, with the help of $h'''$, that 
$\overline{\beta}(n+k+1) \in A$, i.e.,
$\exists j\exists i (\overline{\overline{\beta}(n+k+1)}j = \pi(i))$\footnote{The proof term $a_1 (\snd{\snd{h''}})$ proves 
$\overline{\alpha}n*\seq{0}\in B$, from which 
$\overline\beta(n+1) \in B$ follows using equality axioms.
As remarked earlier, equality-rewriting is implicit in the proof term.}.
Now, one concludes $\beta\in A$ with $(\min(n+k+1,j), h^4)$
by appropriately choosing the witness $\min(n+k+1,j)$
so that $\overline{\overline{\beta}(n+k+1)}j = 
\overline{\beta}(\min(n+k+1,j))$ holds. 
(Again, we suppress the proof term for this equality.)

The proof term $a_1$ derives $\overline\beta(n+1) \in B$
from $\alpha(n)=1$ by making a case distinction. 
To generate the disjunction needed for the case analysis, 
one uses a proof term $a_B$ for $\forall x^{\boole} (x=0 \vee x=1)$.
For the first case
in which $\chi(\overline\beta(n+1)) = 0$, we have an absurdity
$1 = 0$, by definition of $\alpha$, since $\alpha(n)=1$. 
Hence, by equality-rewriting we may use 
the proof term $h_1$ at type $\chi(\overline\beta(n+1)) = 1$.
Now, both the two cases are closed by applying 
$\fst{(b(\overline\beta(n+1)))}$, which proves
$\chi(\overline\beta(n+1)) = 1 \to 
\overline\beta(n+1) \in B$, to $h_1$ and $h_2$, respectively. 

From $\alpha \in A$, one obtains the length $l$ and the index $m$ such that $\overline \alpha l$ is covered by the basic open $\pi(m)$ (the proof term $d$ in line 9), and then one can show that $\overline\alpha 0 = \seq{}$ is in $B$. This last fact is derived by the proof term
\[
a_I\, (\lambda h. h)\, a_3\, l\, (0, \tilde\lambda q.  (l,(m,d))),
\]
where $a_I$ is a proof term behind an instance of the induction axiom showing $\forall l^0 (\overline\alpha l\in B \to \seq{}\in B)$. The proof term $a_I$ uses the proof term $a_3$ which derives
\[
\forall n ((\overline\alpha\, n \in B \to \seq{} \in B)
\to \overline\alpha\, (n+1) \in B \to \seq{} \in B).
\]
It is proved by case analysis, considering the possibilities for the pair $(\chi(\overline\alpha n*\seq{0}), \chi(\overline\alpha n*\seq{1}))$.
If either $\chi(\overline\alpha n *\seq{0}) = 0$ or
$\chi(\overline\alpha n *\seq{1}) = 0$ holds, we close the case
by the characteristic property of $\chi$ 
together with the hypothesis $h$. Otherwise, i.e.
both $\chi(\overline\alpha n *\seq{0}) = 1$ and 
$\chi(\overline\alpha n *\seq{1}) = 1$ holds, we can deduce
$\overline{\alpha}n \in B$ (the proof term $a_4$), 
from which the case follows by
the induction hypothesis. 

\section{Conclusion}\label{sec:conclusion}

We gave a direct proof for \OIC\ in a constructive predicate logic incorporating delimited control operators. 
While computational interpretation of $\mqcplusS{S}$ is available, namely the standard call-by-value weak-head reduction semantics for lambda calculus with shift and reset, we cannot directly analyze the computational behavior of the proof term for \OIC\ because, at the moment,  we do not have a proof term for \ACnatbool\ 
used in the proof term for \OIC.  The best way to overcome this limitation would be to extend $\mqcplusS{S}$ 
so that it can derive \ACnatbool\ as it is done in Martin-Löf Type Theory or constructive versions of Hilbert's epsilon calculus. 

Another way to overcome the limitation would be to use a realizability or functional interpretation that extracts programs from constructive proofs even in presence of choice axioms. For example, by using Spector's extension of Gödel's functional interpretation with bar recursion, we could extract a program from our proof.  
However,  to replace bar recursion is the point of using delimited control operators in the first place. 

If and when our future work is successful, it would allow, at least for the case of the compact Cantor space, to replace Berger's general-recursive computation schema of \emph{open recursion} by a terminating computation schema based on control operators.

The work of Krivine on Classical Realizability gives an interpretation of the Axiom of Dependent Choice \cite{Krivine2003} using control operators for classical logic. Herbelin recently gave a more direct version of that work \cite{Herbelin2012}, using classical control operators and coinduction.

Finally, we would like to mention Veldman's recent work in Constructive Reverse Mathematics \cite{Veldman2011,Veldman2013} that has served as inspiration for our work. An article of Veldman on the equivalence of Open Induction with a number of other axioms is in preparation. In our paper, we showed one direction of this equivalence for the topology of Cantor space seen as the infinite binary tree rather than as the subset of the real line.

\subsection*{Acknowledgments} We would like to thank Wim Veldman for explaining us some of his results, and Ralph Matthes and Hugo Herbelin for valuable comments on the draft. 

\bibliographystyle{plain}
\bibliography{oi-delcont}

\end{document}
